\documentclass[12pt, reqno]{amsart}
\usepackage{ amsmath,amsthm, amscd, amsfonts, amssymb, graphicx, color}
\usepackage[bookmarksnumbered, colorlinks, plainpages,linkcolor=blue,urlcolor=blue,citecolor=blue]{hyperref}
\usepackage{setspace}
\usepackage{ulem}

\newtheorem{thm}{Theorem}[section]
\newtheorem{cor}[thm]{Corollary}
\newtheorem{lem}[thm]{Lemma}

\newtheorem{exam}[thm]{Example}
\numberwithin{equation}{section}

\begin{document}

\title{g-Drazin inverse and group inverse for the anti-triangular block-operator matrices}

\author{Huanyin Chen}
\author{Marjan Sheibani}
\address{
School of Mathematics\\ Hangzhou Normal University\\ Hangzhou, China}
\email{<huanyinchen@aliyun.com>}
\address{Farzanegan Campus, Semnan University, Semnan, Iran}
\email{<m.sheibani@semnan.ac.ir>}

\subjclass[2010]{15A09, 47A08, 65F05.} \keywords{g-Drazin inverse; group inverse; block operator matrix; identical subblock; Banach space.}

\begin{abstract} We present the generalized Drazin inverse for certain anti-triangular operator matrices. Let $E,F,EF^{\pi}\in \mathcal{B}(X)^d$. If $EFEF^{\pi}=0$ and $F^2EF^{\pi}=0$, we prove that $M=\left(
  \begin{array}{cc}
    E&I\\
    F&0
  \end{array}
\right)$ has g-Drazin inverse and its explicit representation is established. Moreover, necessary and sufficient conditions are given for the existence of the group inverse of $M$ under the condition $FEF^{\pi}=0$. The group inverse for the anti-triangular block-operator matrices with two identical subblocks is thereby investigated.
These extend the results of  Zhang and Mosi\'c (Filomat, 32(2018), 5907--5917) and Zou, Chen and Mosi\'c (Studia Scient. Math. Hungar., 54(2017), 489--508).
\end{abstract}

\maketitle

\section{Introduction}

Let $\mathcal{B}(X)$ be a Banach algebra of all bounded linear operators over a Banach space $X$. An operator $T$ in $\mathcal{B}(X)$ has generalized Drazin inverse (i.e., g-Drazin inverse) provided that there exists some $S\in \mathcal{B}(X)$ such that $ST=TS, S=STS, T-T^2S$ is quasinilpotent. Such $S$ is unique, if it exists, and we denote it by $T^d$. If we replace the quasinilpotent by nilpotent in $\mathcal{B}(X)$, $S$ is called the Drazin inverse of $a$, denoted by $T^D$. Let $E,F\in \mathcal{B}(X)$ and $I$ be the identity operator over the Banach space $X$. It is attractive to investigate the g-Drazin inverse of the block-operator matrix $M=\left(
  \begin{array}{cc}
    E&I\\
    F&0
  \end{array}
\right)$. It was firstly posed by Campbell that the solutions to singular systems of differential equations is determined by the Drazin inverse of the preceding operator matrix $M$ (see~\cite{C}).

In 2005, Castro-Gonz\'{a}lez and Dopazo gave the representations of the Drazin inverse for a class of operator matrix $\left(
  \begin{array}{cc}
    I&I\\
    F&0
  \end{array}
\right)$ (see~\cite[Theorem 3.3]{CD}). In 2011, Bu et al. investigate the Drazin inverse of the preceding operator matrix $M$ under the condition $EF=FE$ (see~\cite[Theorem 3.2]{B}). In 2016, Zhang investigated the g-Drazin invertibility of $M$ under the conditions $F^dEF^{\pi}=0, F^{\pi}FE=0$ and $F^{\pi}EF^d=0, EFF^{\pi}=0$ (see~\cite[Theorem 2.6, Theorem 2.8]{Zh1}). In ~\cite[Theorem 2.3]{Zhang},
 Zhang and Mosi\'c considered the g-Drazin inverse of $M$ under the wider condition $FEF^{\pi}=0$. We refer the reader to ~\cite{Z0,Z1,Z2} for further recent progresses on the Drazin and g-Drazin inverse of $M$.

The motivation of this paper is to further study the generalized inverse of the block-operator matrix $M$ under the condition $FEF^{\pi}=0$.
In Section 2, we present the g-Drazin inverse for the operator matrix $M$ under the condition $EFEF^{\pi}=0$ and $F^2EF^{\pi}=0$, which extend ~\cite[Theorem 2.3]{Zhang} to a wider case. As the Drazin and g-Drazin inverses coincide with each other for a complex matrix, our results indeed provide algebraic method to find all function solutions of a new class of singular differential equations posed by Campbell (see~\cite{C}).

An operator $T$ in $\mathcal{B}(X)$ has Drazin inverse if and only if there exists some $S\in \mathcal{B}(X)$ such that $ST=TS, S=STS, T^n=T^{n+1}S$ for some $n\in {\Bbb N}$. Such $S$ is unique, if it exists, and we denote it by $T^D$. Such smallest $n$ is called the Drazin index of $T$ and denote it by $ind(T)$. If $T$ has Drazin index $1$, $T$ is said to have inverse $S$, and denote its group inverse by $T^{\#}$. The group inverse of the block-operator matrices over a Banach space has interesting applications of resistance distances to the bipartiteness of graphs (see~\cite{SW}). Many authors have studied such problems from many different views, e.g., ~\cite{B1,C1,C2,C3,MD2,M,Z2}.

In~\cite[Theorem 2.10]{Z}, Zou et al. studied the group inverse for $M$ under the condition $EF=0$. In Section 3, we present necessary and sufficient condition on
the existence of the group inverse and its representation for the block operator matrix $M$ under the condition $FEF^{\pi}=0$.

In~\cite[Theorem 5]{C2}, Cao et al. considered the group inverse for a block matrix with identical subblocks over a right Ore domain. Finally, in the last section,
we further investigate the necessary and sufficient condition for the existence of the the group inverse of a $2\times 2$ block-operator matrix $\left(
\begin{array}{cc}
E&F\\
F&0
\end{array}
\right)$ with identical subblocks. The explicit formula of its group inverse is thereby given under the same condition $FEF^{\pi}=0$.

Throughout this paper, ${\Bbb C}^{n\times n}$ denotes the Banach algebra of all $n\times n$ complex matrices. Let $X$ be a Banach space. We use $\mathcal{B}(X)$ to denote the Banach algebra of bounded linear operator on $X$. Let $T\in \mathcal{B}(X)^d$. We use $T^{\pi}$ to stand for the spectral idempotent operator $I-TT^d$. Let $p^2=p,X\in \mathcal{B}(X)$. Then we write $X=pXp+pX(1-p)+(1-p)Xp+(1-p)X(1-p),$
and induce a Pierce representation given by the matrix
$X=\left(\begin{array}{cc}
pXp&pXp^{\pi}\\
p^{\pi}Xp&p^{\pi}xp^{\pi}
\end{array}
\right)_p.$

\section{g-Drazin inverse of anti-triangular block matrices}

The aim of this section is to investigate the g-Drazin invertibility of the block-operator matrix $M=\left(
  \begin{array}{cc}
    E&I\\
    F&0
  \end{array}
\right)$. The representation of the g-Drazin inverse of $M$ is given under some new kind of conditions. We begin with

\begin{lem} Let $p^2=p\in \mathcal{B}(X)$ and $X=\left(\begin{array}{cc}
A&0\\
C&B
\end{array}
\right)_p$, where $A,B\in \mathcal{B}(X)^d$ and $C\in \mathcal{B}(X)$. Then $X\in \mathcal{B}(X)^d$ and $X^d=\left(\begin{array}{cc}
A^d&0\\
Z&B^d
\end{array}
\right)_p,$ where $Z=\sum\limits_{i=0}^{\infty}(B^d)^{i+2}CA^iA^{\pi}+\sum\limits_{i=0}^{\infty}B^iB^{\pi}C(A^d)^{i+2}-B^dCA^d.$
\end{lem}
\begin{proof} See ~\cite[Theorem 5.1]{DS}.\end{proof}

\begin{lem} Let $P,Q\in \mathcal{B}(X)^d$. If $PQP=0$ and $Q^2P=0$, then $P+Q\in \mathcal{B}(X)^d$. In this case,
$$\begin{array}{lll}
(P+Q)^d&=&\sum\limits_{i=0}^{\infty}(P+Q)(P^d)^{i+2}Q^iQ^{\pi}+\sum\limits_{i=0}^{\infty}P^{i+1}P^{\pi}(Q^d)^{i+2}\\
&+&\sum\limits_{i=0}^{\infty}QP^iP^{\pi}(Q^d)^{i+2}-(P+Q)P^dQ^d.
\end{array}$$
\end{lem}
\begin{proof} See ~\cite[Theorem 4.2.2]{G}.\end{proof}

We are now ready to prove:

\begin{thm} Let $E,F\in \mathcal{B}(X)^d$. If $EFE=0$ and $F^2E=0$, then $M=\left(
  \begin{array}{cc}
    E&I\\
    F&0
  \end{array}
\right)$ has g-Drazin inverse. In this case, $M^d=\left(
\begin{array}{cc}
\Lambda&\Sigma\\
\Gamma&\Delta
\end{array}
\right),$ where
$$\begin{array}{lll}
\Lambda&=&EE^{\pi}F^d-FE^dF^d+\sum\limits_{i=0}^{\infty}[I+F(E^d)^2](E^d)^{2i+1}F^iF^{\pi}\\
&+&\sum\limits_{i=0}^{\infty}E^{2i+3}E^{\pi}(F^d)^{i+2}+\sum\limits_{i=0}^{\infty} FE^{2i+1}E^{\pi}(F^d)^{i+2},\\
\Sigma&=&-EE^dF^d-F(E^d)^2F^d+\sum\limits_{i=0}^{\infty}[I+F(E^d)^2](E^d)^{2i+2}F^iF^{\pi}\\
&+&\sum\limits_{i=0}^{\infty}E^{2i+2}E^{\pi}(F^d)^{i+2}+\sum\limits_{i=0}^{\infty}FE^{2i}E^{\pi}(F^d)^{i+2},\\
\Gamma&=&FE^{\pi}F^d+\sum\limits_{i=0}^{\infty}F(E^d)^{2i+2}F^iF^{\pi}+\sum\limits_{i=0}^{\infty}FE^{2i+2}E^{\pi}(F^d)^{i+2},\\
\Delta&=&-FE^dF^d+\sum\limits_{i=0}^{\infty}F(E^d)^{2i+3}F^iF^{\pi}+\sum\limits_{i=0}^{\infty}FE^{2i+1}E^{\pi}(F^d)^{i+2}.
\end{array}$$\end{thm}
\begin{proof} Clearly, we have $$M^2=\left(\begin{array}{cc}
E^2+F&E\\
FE&F
\end{array}
\right)=P+Q,$$ where $$P=\left(\begin{array}{cc}
E^2&E\\
0&0
\end{array}
\right), Q=\left(\begin{array}{cc}
F&0\\
FE&F
\end{array}
\right).$$ Using Lemma 2.1, we have
$$\begin{array}{l}
         Q^d=\left(\begin{array}{cc}
F^d&0\\
FE(F^d)^{2}&F^d
\end{array}
\right), Q^{\pi}=\left(\begin{array}{cc}
F^{\pi}&0\\
-FEF^d&F^{\pi}
\end{array}
\right).
\end{array}$$ Likewise, we obtain
$P^d=\left(
         \begin{array}{cc}
          (E^d)^2 &(E^d)^3 \\
          0 & 0 \\
          \end{array}
         \right), P^{\pi}=\left(
         \begin{array}{cc}
          E^{\pi}&-E^d\\
          0 & I \\
          \end{array}
         \right).$
         One easily checks that $$\begin{array}{lll}
PQP&=&\left(\begin{array}{cc}
E^2F&EF\\
0&0
\end{array}
\right)\left(\begin{array}{cc}
E^2&E\\
0&0
\end{array}
\right)\\
&=&\left(\begin{array}{cc}
E^2FE^2&E^2FE\\
0&0
\end{array}
\right)\\
&=&0;\\
Q^2P&=&\left(\begin{array}{cc}
F&0\\
FE&F
\end{array}
\right)\left(\begin{array}{cc}
FE^2&FE\\
0&0
\end{array}
\right)\\
&=&\left(\begin{array}{cc}
F^2E^2&F^2E\\
FEFE^2&FEFE
\end{array}
\right)\\
&=&0.
\end{array}$$ In light of Lemma 2.2, $M^2$ has g-Drazin inverse, and so $M$ has g-Drazin inverse. In this case, $$\begin{array}{lll}
M^d&=&M(M^2)^d\\
&=&\sum\limits_{i=0}^{\infty}M^3(P^d)^{i+2}Q^iQ^{\pi}+\sum\limits_{i=0}^{\infty}MP^{i+1}P^{\pi}(Q^d)^{i+2}\\
&+&\sum\limits_{i=0}^{\infty}MQP^iP^{\pi}(Q^d)^{i+2}-M^3P^dQ^d.
\end{array}$$
We compute that
$$\begin{array}{ll}
&M^3(P^d)^{i+2}Q^iQ^{\pi}\\
=&\left(
  \begin{array}{cc}
    E&I\\
    F&0
  \end{array}
\right)^3\left(
         \begin{array}{cc}
          (E^d)^2 &(E^d)^3 \\
          0 & 0 \\
          \end{array}
         \right)^{i+2}\left(\begin{array}{cc}
F&0\\
FE&F
\end{array}
\right)^i\left(\begin{array}{cc}
F^{\pi}&0\\
-FEF^d&F^{\pi}
\end{array}
\right)\\
=&\left(
  \begin{array}{cc}
    E^3+EF+FE&E^2+F\\
    FE^2+F^2&FE
  \end{array}
\right)\left(
         \begin{array}{cc}
          (E^d)^{2i+4} &(E^d)^{2i+5}\\
          0 & 0 \\
          \end{array}
         \right)\\
        &\left(\begin{array}{cc}
F^i&0\\
FEF^{i-1}&F^i
\end{array}
\right)\left(\begin{array}{cc}
F^{\pi}&0\\
-FEF^d&F^{\pi}
\end{array}
\right)\\
=&\left(
  \begin{array}{cc}
    (E^d)^{2i+1}+F(E^d)^{2i+3}&(E^d)^{2i+2}+F(E^d)^{2i+4}\\
    F(E^d)^{2i+2}&F(E^d)^{2i+3}
  \end{array}
\right)\\
        &\left(\begin{array}{cc}
F^iF^{\pi}&0\\
FEF^{i-1}F^{\pi}&F^iF^{\pi}
\end{array}
\right)\\
=&\left(
  \begin{array}{cc}
  [I+F(E^d)^2](E^d)^{2i+1}F^iF^{\pi}&[I+F(E^d)^2](E^d)^{2i+2}F^iF^{\pi}\\
    F(E^d)^{2i+2}F^iF^{\pi}&F(E^d)^{2i+3}F^iF^{\pi}
  \end{array}
\right) (i\geq 1),\\
\end{array}$$
$$\begin{array}{ll}
&M^3(P^d)^2Q^{\pi}\\
=&\left(
  \begin{array}{cc}
    E&I\\
    F&0
  \end{array}
\right)^3\left(
         \begin{array}{cc}
          (E^d)^2 &(E^d)^3 \\
          0 & 0 \\
          \end{array}
         \right)^2\left(\begin{array}{cc}
F^{\pi}&0\\
-FEF^d&F^{\pi}
\end{array}
\right)\\
=&\left(
  \begin{array}{cc}
    E^3+EF+FE&E^2+F\\
    FE^2+F^2&FE
  \end{array}
\right)\left(
         \begin{array}{cc}
          (E^d)^{4} &(E^d)^{5}\\
          0 & 0 \\
          \end{array}
         \right)\left(\begin{array}{cc}
F^{\pi}&0\\
-FEF^d&F^{\pi}
\end{array}
\right)\\
=&\left(
  \begin{array}{cc}
    E^d+F(E^d)^3&(E^d)^2+F(E^d)^4\\
    F(E^d)^2&F(E^d)^3
  \end{array}
\right)\left(\begin{array}{cc}
F^{\pi}&0\\
-FEF^d&F^{\pi}
\end{array}
\right)\\
=&\left(
  \begin{array}{cc}
  [I+F(E^d)^2]E^dF^iF^{\pi}&[I+F(E^d)^2](E^d)^2F^{\pi}\\
    F(E^d)^2F^{\pi}&F(E^d)^3F^{\pi}
  \end{array}
\right),\\
\end{array}$$
$$\begin{array}{ll}
&MP^{i+1}P^{\pi}(Q^d)^{i+2}\\
=&\left(
  \begin{array}{cc}
    E&I\\
    F&0
  \end{array}
\right)\left(\begin{array}{cc}
E^2&E\\
0&0
\end{array}
\right)^{i+1}
\left(
         \begin{array}{cc}
          E^{\pi}&-E^d\\
          0 & I \\
          \end{array}
         \right) \left(\begin{array}{cc}
F^d&0\\
FE(F^d)^{2}&F^d
\end{array}
\right)^{i+2}\\
=&{\tiny\left(
\begin{array}{cc}
E^{2i+3}E^{\pi}&-E^{2i+3}E^d+E^{2i+2}\\
FE^{2i+2}E^{\pi}&-FE^{2i+2}E^d+FE^{2i+1}\\
\end{array}
\right)
\left(\begin{array}{cc}
(F^d)^{i+2}&0\\
FE(F^d)^{i+3}&(F^d)^{i+2}
\end{array}
\right)}\\
=& \left(\begin{array}{cc}
E^{2i+3}E^{\pi}(F^d)^{i+2}&E^{2i+2}E^{\pi}(F^d)^{i+2}\\
FE^{2i+2}E^{\pi}(F^d)^{i+2}&FE^{2i+1}E^{\pi}(F^d)^{i+2}
\end{array}
\right),\\
&\\
\end{array}$$
$$\begin{array}{ll}
&MQP^{\pi}(Q^d)^{2}\\
=&\left(
  \begin{array}{cc}
    E&I\\
    F&0
  \end{array}
\right)\left(\begin{array}{cc}
F&0\\
FE&F
\end{array}
\right)\left(
         \begin{array}{cc}
          E^{\pi}&-E^d\\
          0 & I \\
          \end{array}
         \right)
\left(\begin{array}{cc}
(F^d)^2&0\\
FE(F^d)^{3}&(F^d)^2
\end{array}
\right)\\
=&\left(\begin{array}{cc}
EFE^{\pi}(F^d)^2+FEE^{\pi}(F^d)^2&F^d-FEE^d(F^d)^2\\
FF^d&0
\end{array}
\right),
\end{array}$$
$$\begin{array}{ll}
&MQP^iP^{\pi}(Q^d)^{i+2}\\
=&{\tiny\left(
  \begin{array}{cc}
    EF+FE&F\\
    F^2&0
  \end{array}
\right)\left(\begin{array}{cc}
E^2&E\\
0&0
\end{array}
\right)^{i}\left(
         \begin{array}{cc}
          E^{\pi}&-E^d\\
          0 & I \\
          \end{array}
         \right)
\left(\begin{array}{cc}
F^d&0\\
FE(F^d)^{2}&F^d
\end{array}
\right)^{i+2}}\\
=&\left(\begin{array}{cc}
FE^{2i+1}E^{\pi}(F^d)^{i+2}&FE^{2i}E^{\pi}(F^d)^{i+2}\\
0&0
\end{array}
\right) (i\geq 1),\\
\end{array}$$
$$\begin{array}{ll}
&M^3P^dQ^d\\
=&\left(
  \begin{array}{cc}
    E^3+EF+FE&E^2+F\\
    FE^2+F^2&FE
  \end{array}
\right)\left(
         \begin{array}{cc}
          (E^d)^2F^d&(E^d)^3F^d\\
          0 & 0 \\
          \end{array}
         \right)\\
=&\left(
  \begin{array}{cc}
    E^2E^dF^d+FE^dF^d&EE^dF^d+F(E^d)^2F^d\\
    FEE^dF^d&FE^dF^d
  \end{array}
\right).
\end{array}$$ Therefore
$M^d=\left(
\begin{array}{cc}
\Lambda&\Sigma\\
\Gamma&\Delta
\end{array}
\right),$ where $\Lambda,\Sigma,\Gamma,\Delta$ are as in the preceding stating. This completes the proof.\end{proof}

\begin{lem} Let $P,Q\in \mathcal{B}(X)^{d}$. If $PQ=0$, then $P+Q\in \mathcal{B}(X)^{d}$. In this case, $$(P+Q)^d=\sum\limits_{i=0}^{\infty}Q^iQ^{\pi}(P^d)^{i+1}+\sum\limits_{i=0}^{\infty}(Q^d)^{i+1}P^iP^{\pi}.$$
\end{lem}
\begin{proof} See ~\cite[Theorem 2.3]{DW}.\end{proof}

We come now to prove the main result of this section, which is an extension of ~\cite[Theorem 2.3]{Zhang} for block-operator matrices.

\begin{thm} Let $E,F,EF^{\pi}\in \mathcal{B}(X)^d$. If $EFEF^{\pi}=0$ and $F^2EF^{\pi}=0$, then $M=\left(
  \begin{array}{cc}
    E&I\\
    F&0
  \end{array}
\right)$ has g-Drazin inverse. In this case, $$M^d={\small\left(
\begin{array}{cc}
\varepsilon&[\zeta-(\alpha\zeta+\beta\theta)]\delta^d+\sum\limits_{i=1}^{\infty}[\zeta_{i}(1-\gamma\zeta)-\varepsilon_{i}(\alpha\zeta+\beta\theta)](\delta^d)^{i+1}\\
\eta&[\theta+(1-\gamma\zeta)]\delta^d+\sum\limits_{i=1}^{\infty}[\theta_{i}(1-\gamma\zeta)-\eta_{i}(\alpha\zeta+\beta\theta)](\delta^d)^{i+1}
\end{array}
\right)},$$ where
$$\begin{array}{lll}
\alpha&=&EF^{\pi}, \beta=F^{\pi}EFF^d+F^{\pi},\gamma=FF^{\pi},\\
\delta^d&=&F^d+FF^d-FF^dEF^d;
\end{array}$$
$$\begin{array}{lll}
\varepsilon&=&(\alpha\Lambda+\Gamma)\Lambda+(\alpha\Sigma+\Delta)\Gamma,\\
\zeta &=& (\alpha\Lambda+\Gamma)\Sigma\beta+(\alpha\Sigma+\Delta)\Delta\beta,\\
\eta&=&\gamma\Lambda^2+\gamma\Sigma\Gamma,\\
\theta &=&\gamma\Lambda\Sigma\beta+\gamma\Sigma\Delta\beta;
\end{array}$$
$$\begin{array}{lll}
\varepsilon_{n+1}&=&\alpha\varepsilon_n+\beta\eta_n, \varepsilon_{1}=\varepsilon,\\
\zeta_{n+1}&=&\alpha\zeta_n+\beta\theta_n,\zeta_{1}=\zeta,\\
\eta_{n+1}&=&\gamma\varepsilon_n,\eta_{1}=\eta,\\
\theta_{n+1}&=&\gamma\theta_n,\theta_{1}=\theta;
\end{array}$$
$$\begin{array}{lll}
\Lambda&=&\sum\limits_{i=0}^{\infty}[1+\beta\gamma(\alpha^d)^2](\alpha^d)^{2i+1}(\beta\gamma)^i,\\
\Sigma&=&\sum\limits_{i=0}^{\infty}[1+\beta\gamma(\alpha^d)^2](\alpha^d)^{2i+2}(\beta\gamma)^i,\\
\Gamma&=&\sum\limits_{i=0}^{\infty}\beta\gamma(\alpha^d)^{2i+2}(\beta\gamma)^i,\\
\Delta&=&\sum\limits_{i=0}^{\infty}\beta\gamma(\alpha^d)^{2i+3}(\beta\gamma)^i.
\end{array}$$
\end{thm}
\begin{proof} Let $p=\left(
\begin{array}{cc}
F^{\pi}&0\\
0&0
\end{array}
\right).$ Then $M=\left(
\begin{array}{cc}
\alpha&\beta\\
\gamma&\delta
\end{array}
\right)_p,$ where
$$\begin{array}{rll}
\alpha&=&\left(
\begin{array}{cc}
EF^{\pi}&0\\
0&0
\end{array}
\right), \beta=\left(
\begin{array}{cc}
F^{\pi}EFF^d&F^{\pi}\\
0&0
\end{array}
\right),\\
\gamma&=&\left(
\begin{array}{cc}
0&0\\
FF^{\pi}&0
\end{array}
\right), \delta=\left(
\begin{array}{cc}
FF^dE&FF^d\\
F^2F^d&0
\end{array}
\right).
\end{array}$$
Then $M=P+Q$, where $P=\left(
\begin{array}{cc}
0&0\\
0&\delta
\end{array}
\right), Q=\left(
\begin{array}{cc}
\alpha&\beta\\
\gamma&0
\end{array}
\right).$
Since $F^2EF^{\pi}=0$, we have $FF^dEF^{\pi}=(F^d)^2F^2EF^{\pi}=0$. By hypothesis, $E,EF^{\pi}$ have g-Drazin inverses.
In view of~\cite[Lemma 2.2]{Zhang}, $FF^dE$ has g-Drazin inverse and $(FF^dE)^d=FF^dE^d$. By using ~\cite[Lemma 2.2]{Zhang} again, $EF^{\pi}$ has g-Drazin inverse and $(EF^{\pi})^d=E^dF^{\pi}$. Moreover, $FF^dE^dF^{\pi}=FF^d(EF^{\pi})^d=(F^d)^2F^2EF^{\pi}[(EF^{\pi})^d]^2=0$, and then
$F^{\pi}E^dF^{\pi}=E^dF^{\pi}$. Hence, $\alpha$ has g-Drazin inverse and $$\alpha^d=\left(
\begin{array}{cc}
F^{\pi}E^dF^{\pi}&0\\
0&0
\end{array}
\right), \alpha^{\pi}=\left(
\begin{array}{cc}
F^{\pi}E^{\pi}F^{\pi}&0\\
0&0
\end{array}
\right).$$ One easily checks that
$\beta\gamma=\left(
\begin{array}{cc}
FF^{\pi}&0\\
0&0
\end{array}
\right), (\beta\gamma)^d=0.$ Obviously, we have
$$\begin{array}{rll}
\delta^d&=&\left(
\begin{array}{cc}
0&F^d\\
FF^d&-FF^dEF^d
\end{array}
\right), \delta^{\pi}=\left(
\begin{array}{cc}
0&0\\
0&F^{\pi}
\end{array}
\right);\\
P^d&=&\left(
\begin{array}{cc}
0&0\\
0&\delta^d
\end{array}
\right), P^{\pi}=\left(
\begin{array}{cc}
p&0\\
0&\delta^{\pi}
\end{array}
\right),P^iP^{\pi}=0~(i\geq 1).
\end{array}$$
We compute that $\alpha (\beta\gamma)\alpha=0$ and $(\beta\gamma)^2\alpha =0.$ According to Theorem 2.3,
$\left(
\begin{array}{cc}
\alpha&1\\
\beta\gamma&0
\end{array}
\right)$ has g-Drazin inverse and
$\left(
\begin{array}{cc}
\alpha&1\\
\beta\gamma&0
\end{array}
\right)^d=\left(
\begin{array}{cc}
\Lambda&\Sigma\\
\Gamma&\Delta
\end{array}
\right),$ where
$$\begin{array}{lll}
\Lambda&=&\alpha (\beta\gamma)^d+\beta\gamma\alpha\alpha^{\pi}[(\beta\gamma)^d]^2-\alpha^2\alpha^d(\beta\gamma)^d-\beta\gamma\alpha^d(\beta\gamma)^d\\
&+&\sum\limits_{i=0}^{\infty}[I+\beta\gamma(\alpha^d)^2](\alpha^d)^{2i+1}(\beta\gamma)^i(\beta\gamma)^{\pi}\\
&+&\sum\limits_{i=0}^{\infty}\alpha^{2i+3}\alpha^{\pi}[(\beta\gamma)^d]^{i+2}+\sum\limits_{i=1}^{\infty} \beta\gamma\alpha^{2i+1}\alpha^{\pi}[(\beta\gamma)^d]^{i+2},\\
\Sigma&=&(\beta\gamma)^d-\beta\gamma\alpha\alpha^d[(\beta\gamma)^d]^2-\alpha\alpha^d(\beta\gamma)^d-\beta\gamma(\alpha^d)^2(\beta\gamma)^d\\
&+&\sum\limits_{i=0}^{\infty}[1+\beta\gamma(\alpha^d)^2](\alpha^d)^{2i+2}(\beta\gamma)^i(\beta\gamma)^{\pi}\\
&+&\sum\limits_{i=0}^{\infty}\alpha^{2i+2}\alpha^{\pi}[(\beta\gamma)^d]^{i+2}+\sum\limits_{i=1}^{\infty}\beta\gamma\alpha^{2i}\alpha^{\pi}[(\beta\gamma)^d]^{i+2},\\
\Gamma&=&\beta\gamma (\beta\gamma)^d-\beta\gamma\alpha\alpha^d(\beta\gamma)^d+\sum\limits_{i=0}^{\infty}\beta\gamma(\alpha^d)^{2i+2}(\beta\gamma)^i(\beta\gamma)^{\pi}\\
&+&\sum\limits_{i=0}^{\infty}\beta\gamma\alpha^{2i+2}\alpha^{\pi}[(\beta\gamma)^d]^{i+2},\\
\Delta&=&-\beta\gamma\alpha^d(\beta\gamma)^d+\sum\limits_{i=0}^{\infty}\beta\gamma(\alpha^d)^{2i+3}(\beta\gamma)^i(\beta\gamma)^{\pi}\\
&+&\sum\limits_{i=0}^{\infty}\beta\gamma\alpha^{2i+1}\alpha^{\pi}[(\beta\gamma)^d]^{i+2}.
\end{array}$$
Thus, we derive
$$\begin{array}{lll}
\Lambda&=&\sum\limits_{i=0}^{\infty}[1+\beta\gamma(\alpha^d)^2](\alpha^d)^{2i+1}(\beta\gamma)^i,\\
\Sigma&=&\sum\limits_{i=0}^{\infty}[1+\beta\gamma(\alpha^d)^2](\alpha^d)^{2i+2}(\beta\gamma)^i,\\
\Gamma&=&\sum\limits_{i=0}^{\infty}\beta\gamma(\alpha^d)^{2i+2}(\beta\gamma)^i,\\
\Delta&=&\sum\limits_{i=0}^{\infty}\beta\gamma(\alpha^d)^{2i+3}(\beta\gamma)^i.
\end{array}$$
We compute that
$$\begin{array}{ll}
&(1+\beta\gamma(\alpha^d)^2)(\alpha^d)^{2i+1}(\beta\gamma)^i\\
=&\left(
\begin{array}{cc}
I+F(F^{\pi}E^dF^{\pi})^2&0\\
0&I
\end{array}
\right)\left(
\begin{array}{cc}
(F^{\pi}E^dF^{\pi})^{2i+1}&0\\
0&0
\end{array}
\right)\left(
\begin{array}{cc}
F^iF^{\pi}&0\\
0&0
\end{array}
\right)\\
=&\left(
\begin{array}{cc}
[I+F(F^{\pi}E^dF^{\pi})^2](F^{\pi}E^dF^{\pi})^{2i+1}F^i&0\\
0&0
\end{array}
\right),\\
\end{array}$$
$$\begin{array}{ll}
&(1+\beta\gamma(\alpha^d)^2)(\alpha^d)^{2i+2}(\beta\gamma)^i\\
=&\left(
\begin{array}{cc}
I+F(F^{\pi}E^dF^{\pi})^2&0\\
0&I
\end{array}
\right)\left(
\begin{array}{cc}
(F^{\pi}E^dF^{\pi})^{2i+2}&0\\
0&0
\end{array}
\right)\left(
\begin{array}{cc}
F^iF^{\pi}&0\\
0&0
\end{array}
\right)\\
=&\left(
\begin{array}{cc}
[I+F(F^{\pi}E^dF^{\pi})^2](F^{\pi}E^dF^{\pi})^{2i+2}F^i&0\\
0&0
\end{array}
\right),\\
&\beta\gamma(\alpha^d)^{2i+2}(\beta\gamma)^i\\
=&\left(
\begin{array}{cc}
FF^{\pi}&0\\
0&0
\end{array}
\right)\left(
\begin{array}{cc}
(F^{\pi}E^dF^{\pi})^{2i+2}&0\\
0&0
\end{array}
\right)\left(
\begin{array}{cc}
F^iF^{\pi}&0\\
0&0
\end{array}
\right)\\
=&\left(
\begin{array}{cc}
F(F^{\pi}E^dF^{\pi})^{2i+2}F^iF^{\pi}&0\\
0&0
\end{array}
\right),\\
&\beta\gamma(\alpha^d)^{2i+3}(\beta\gamma)^i\\
=&\left(
\begin{array}{cc}
FF^{\pi}&0\\
0&0
\end{array}
\right)\left(
\begin{array}{cc}
(F^{\pi}E^dF^{\pi})^{2i+3}&0\\
0&0
\end{array}
\right)\left(
\begin{array}{cc}
F^iF^{\pi}&0\\
0&0
\end{array}
\right)\\
=&\left(
\begin{array}{cc}
F(F^{\pi}E^dF^{\pi})^{2i+3}F^iF^{\pi}&0\\
0&0
\end{array}
\right).
\end{array}$$
We verify that
$$\begin{array}{c}
\left(
\begin{array}{cc}
\alpha&\beta\\
\gamma&0
\end{array}
\right)=\left(
\begin{array}{cc}
\alpha&1\\
\gamma&0
\end{array}
\right)
\left(
\begin{array}{cc}
1&0\\
0&\beta
\end{array}
\right),\\
\left(
\begin{array}{cc}
\alpha&1\\
\beta\gamma&0
\end{array}
\right)=\left(
\begin{array}{cc}
1&0\\
0&\beta
\end{array}
\right)
\left(
\begin{array}{cc}
\alpha&1\\
\gamma&0
\end{array}
\right).
\end{array}$$
Therefore it follows by Cline's formula (see~\cite[Theorem 2.2]{L}) that
$$\begin{array}{lll}
Q^d&=&\left(
\begin{array}{cc}
\alpha&1\\
\gamma&0
\end{array}
\right)\left(
\begin{array}{cc}
\Lambda&\Sigma\\
\Gamma&\Delta
\end{array}
\right)^2
\left(
\begin{array}{cc}
1&0\\
0&\beta
\end{array}
\right)\\
&=&\left(
\begin{array}{cc}
\alpha\Lambda+\Gamma&\alpha\Sigma+\Delta\\
\gamma\Lambda&\gamma\Sigma
\end{array}
\right)\left(
\begin{array}{cc}
\Lambda&\Sigma\beta\\
\Gamma&\Delta\beta
\end{array}
\right)\\
&=&\left(
\begin{array}{cc}
\varepsilon &\zeta\\
\eta &\theta
\end{array}
\right),
\end{array}$$
where $$\begin{array}{lll}
\varepsilon&=&(\alpha\Lambda+\Gamma)\Lambda+(\alpha\Sigma+\Delta)\Gamma,\\
\zeta &=& (\alpha\Lambda+\Gamma)\Sigma\beta+(\alpha\Sigma+\Delta)\Delta\beta,\\
\eta&=&\gamma\Lambda^2+\gamma\Sigma\Gamma,\\
\theta &=&\gamma\Lambda\Sigma\beta+\gamma\Sigma\Delta\beta .
\end{array}$$ Moreover, we have $$\begin{array}{lll}
Q^{\pi}&=&\left(
\begin{array}{cc}
p&0\\
0&1-p
\end{array}
\right)-\left(
\begin{array}{cc}
\alpha&\beta\\
\gamma&0
\end{array}
\right)\left(
\begin{array}{cc}
\varepsilon &\zeta\\
\eta &\theta
\end{array}
\right)\\
&=&\left(
\begin{array}{cc}
p-\alpha\varepsilon-\beta\eta&-\alpha\zeta-\beta\theta\\
-\gamma\varepsilon &1-p-\gamma\zeta
\end{array}
\right).
\end{array}$$
Write $Q^n=\left(
\begin{array}{cc}
\varepsilon_n &\zeta_n\\
\eta_n &\theta_n
\end{array}
\right).$ Then $$\begin{array}{lll}
\varepsilon_{n+1}&=&\alpha\varepsilon_n+\beta\eta_n,\\
\zeta_{n+1}&=&\alpha\zeta_n+\beta\theta_n,\\
\eta_{n+1}&=&\gamma\varepsilon_n,\\
\theta_{n+1}&=&\gamma\theta_n.
\end{array}$$
$$\begin{array}{ll}
&Q^iQ^{\pi}(P^d)^{i+1}\\
=&\left(
\begin{array}{cc}
\varepsilon_{i} &\zeta_{i}\\
\eta_{i}&\theta_{i}
\end{array}
\right)\left(
\begin{array}{cc}
p-\alpha\varepsilon-\beta\eta&-\alpha\zeta-\beta\theta\\
-\gamma\varepsilon &1-p-\gamma\zeta
\end{array}
\right)\left(
\begin{array}{cc}
0&0\\
0&(\delta^d)^{i+1}
\end{array}
\right)\\
=&\left(
\begin{array}{cc}
\varepsilon_{i} &\zeta_{i}\\
\eta_{i}&\theta_{i}
\end{array}
\right)\left(
\begin{array}{cc}
0&-(\alpha\zeta+\beta\theta)(\delta^d)^{i+1}\\
0&(1-p-\gamma\zeta)(\delta^d)^{i+1}
\end{array}
\right)\\
=&\left(
\begin{array}{cc}
0&\zeta_{i}(1-p-\gamma\zeta)(\delta^d)^{i+1}-\varepsilon_{i}(\alpha\zeta+\beta\theta)(\delta^d)^{i+1}\\
0&\theta_{i}(1-p-\gamma\zeta)(\delta^d)^{i+1}-\eta_{i}(\alpha\zeta+\beta\theta)(\delta^d)^{i+1}
\end{array}
\right).
\end{array}$$
Since $\delta\gamma=0$, we have $PQ=0$. In light of Lemma 2.4,
$$\begin{array}{lll}
M^d&=&\sum\limits_{i=0}^{\infty}Q^iQ^{\pi}(P^d)^{i+1}+\sum\limits_{i=0}^{\infty}(Q^d)^{i+1}P^iP^{\pi}\\
&=&Q^dP^{\pi}+Q^{\pi}P^d+\sum\limits_{i=1}^{\infty}Q^iQ^{\pi}(P^d)^{i+1}\\
&=&\left(
\begin{array}{cc}
\varepsilon&\zeta\delta^d\\
\eta&\theta\delta^d
\end{array}
\right)+\left(
\begin{array}{cc}
0&-(\alpha\zeta+\beta\theta)\delta^d\\
0&(1-p-\gamma\zeta)\delta^d
\end{array}
\right)\\
&+&\left(
\begin{array}{cc}
0&\sum\limits_{i=1}^{\infty}[\zeta_{i}(1-p-\gamma\zeta)-\varepsilon_{i}(\alpha\zeta+\beta\theta)](\delta^d)^{i+1}\\
0&\sum\limits_{i=1}^{\infty}[\theta_{i}(1-p-\gamma\zeta)-\eta_{i}(\alpha\zeta+\beta\theta)](\delta^d)^{i+1}
\end{array}
\right)\\
\end{array}$$
$$\begin{array}{l}
{\small =\left(
\begin{array}{cc}
\varepsilon&[\zeta-(\alpha\zeta+\beta\theta)]\delta^d+\sum\limits_{i=1}^{\infty}[\zeta_{i}(1-\gamma\zeta)-\varepsilon_{i}(\alpha\zeta+\beta\theta)](\delta^d)^{i+1}\\
\eta&[\theta+(1-\gamma\zeta)]\delta^d+\sum\limits_{i=1}^{\infty}[\theta_{i}(1-\gamma\zeta)-\eta_{i}(\alpha\zeta+\beta\theta)](\delta^d)^{i+1}
\end{array}
\right).}\end{array}
$$ as required.\end{proof}

\begin{cor} Let $E,F,EF^{\pi}\in \mathcal{B}(X)^d$. If $EFEF^{\pi}=0$ and $F^2EF^{\pi}=0$, then $M=\left(
  \begin{array}{cc}
    E&F\\
    I&0
  \end{array}
\right)$ has g-Drazin inverse. In this case, $$\begin{array}{l}
M^d=\left(
  \begin{array}{cc}
    E&I\\
    I&0
  \end{array}
\right)\\
\left(
\begin{array}{cc}
\varepsilon&[\zeta-(\alpha\zeta+\beta\theta)]\delta^d+\sum\limits_{i=1}^{\infty}[\zeta_{i}(1-\gamma\zeta)-\varepsilon_{i}(\alpha\zeta+\beta\theta)](\delta^d)^{i+1}\\
\eta&[\theta+(1-\gamma\zeta)]\delta^d+\sum\limits_{i=1}^{\infty}[\theta_{i}(1-\gamma\zeta)-\eta_{i}(\alpha\zeta+\beta\theta)](\delta^d)^{i+1}
\end{array}
\right)^2\\
\left(
  \begin{array}{cc}
    I&0\\
    0&F
  \end{array}
\right),
\end{array}$$ where $\alpha,\beta,\gamma,\delta,\varepsilon,\zeta,\eta,\theta,\varepsilon_{n},\zeta_{n}, \eta_{n}, \theta_{n}$ and
$\Lambda,\Sigma,\Gamma,\Delta $ are given as in Theorem 2.5.\end{cor}
\begin{proof} In view of Theorem 2.5, the block operator matrix $\left(
  \begin{array}{cc}
    E&I\\
    F&0
  \end{array}
\right)$ has g-Drazin inverse. We easily see that $$\left(
  \begin{array}{cc}
    E&I\\
    F&0
  \end{array}
\right)=\left(
  \begin{array}{cc}
    I&0\\
    0&F
  \end{array}
\right)\left(
  \begin{array}{cc}
    E&I\\
    I&0
  \end{array}
\right),$$ it follows by Cline's formula that $\left(
  \begin{array}{cc}
    E&I\\
    I&0
  \end{array}
\right)\left(
  \begin{array}{cc}
    I&0\\
    0&F
  \end{array}
\right)$ has g-Drazin inverse. That is, $\left(
  \begin{array}{cc}
    E&F\\
    I&0
  \end{array}
\right)$ has g-Drazin inverse. Moreover, we have
$$\left(
  \begin{array}{cc}
    E&F\\
    I&0
  \end{array}
\right)^d=\left(
  \begin{array}{cc}
    E&I\\
    I&0
  \end{array}
\right)[\left(
  \begin{array}{cc}
    E&I\\
    F&0
  \end{array}
\right)^d]^2\left(
  \begin{array}{cc}
    I&0\\
    0&F
  \end{array}
\right),$$ and so the proof is complete by Theorem 2.5.\end{proof}

We note that the corresponding facts of the preceding lemmas in this section are valid for Drazin inverse. Then by using the most of the technicalities that occur in the proof of Theorem 2.5, we have

\begin{thm} Let $E,F,EF^{\pi}\in \mathcal{B}(X)^D$. If $EFEF^{\pi}=0$ and $F^2EF^{\pi}=0$, then $M=\left(
  \begin{array}{cc}
    E&I\\
    F&0
  \end{array}
\right)$ has Drazin inverse. In this case, $$M^D={\small\left(
\begin{array}{cc}
\varepsilon&[\zeta-(\alpha\zeta+\beta\theta)]\delta^D+\sum\limits_{i=1}^{k}[\zeta_{i}(1-\gamma\zeta)-\varepsilon_{i}(\alpha\zeta+\beta\theta)](\delta^D)^{i+1}\\
\eta&[\theta+(1-\gamma\zeta)]\delta^D+\sum\limits_{i=1}^{k}[\theta_{i}(1-\gamma\zeta)-\eta_{i}(\alpha\zeta+\beta\theta)](\delta^D)^{i+1}
\end{array}
\right)},$$ where
$$\begin{array}{lll}
\alpha&=&EF^{\pi}, \beta=F^{\pi}EFF^D+F^{\pi},\gamma=FF^{\pi},\\
\delta^D&=&F^D+FF^D-FF^DEF^D;
\end{array}$$
$$\begin{array}{lll}
\varepsilon&=&(\alpha\Lambda+\Gamma)\Lambda+(\alpha\Sigma+\Delta)\Gamma,\\
\zeta &=& (\alpha\Lambda+\Gamma)\Sigma\beta+(\alpha\Sigma+\Delta)\Delta\beta,\\
\eta&=&\gamma\Lambda^2+\gamma\Sigma\Gamma,\\
\theta &=&\gamma\Lambda\Sigma\beta+\gamma\Sigma\Delta\beta;
\end{array}$$
$$\begin{array}{lll}
\varepsilon_{n+1}&=&\alpha\varepsilon_n+\beta\eta_n, \varepsilon_{1}=\varepsilon,\\
\zeta_{n+1}&=&\alpha\zeta_n+\beta\theta_n,\zeta_{1}=\zeta,\\
\eta_{n+1}&=&\gamma\varepsilon_n,\eta_{1}=\eta,\\
\theta_{n+1}&=&\gamma\theta_n,\theta_{1}=\theta;
\end{array}$$
$$\begin{array}{lll}
\Lambda&=&\sum\limits_{i=0}^{m}[1+\beta\gamma(\alpha^D)^2](\alpha^D)^{2i+1}(\beta\gamma)^i,\\
\Sigma&=&\sum\limits_{i=0}^{m}[1+\beta\gamma(\alpha^D)^2](\alpha^D)^{2i+2}(\beta\gamma)^i,\\
\Gamma&=&\sum\limits_{i=0}^{m}\beta\gamma(\alpha^D)^{2i+2}(\beta\gamma)^i,\\
\Delta&=&\sum\limits_{i=0}^{m}\beta\gamma(\alpha^D)^{2i+3}(\beta\gamma)^i,
\end{array}$$ where $k=i(EF^{\pi})+2i(F), m=i(F)$.
\end{thm}

\section{group inverse of anti-triangular block matrices}

The aim of this section is to provide necessary and sufficient conditions on $E$ and $F$ so that the block operator matrix $\left(
\begin{array}{cc}
E&I\\
F&0
\end{array}
\right)$ has group inverse. We now derive

\begin{thm} Let $M=\left(
\begin{array}{cc}
E&I\\
F&0
\end{array}
\right)$ and $E,F,EF^{\pi}$ have Drazin inverses. If $FEF^{\pi}=0$, then the following are equivalent:
\end{thm}
\begin{enumerate}
\item [(1)]{\it $M$ has group inverse.} \vspace{-.5mm}
\item [(2)]{\it $F$ has group inverse and $E^{\pi}F^{\pi}=0$.}\\
\end{enumerate}
In this case, $$M^{\#}=\left(
\begin{array}{cc}
E^DF^{\pi}&F^{\#}+(E^DF^{\pi})^2-E^DF^{\pi}EF^{\#}\\
FF^{\#}&-FF^{\#}EF^{\#}
\end{array}
\right).$$
\begin{proof} $(1)\Rightarrow (2)$ Write $M^{\#}=\left(
\begin{array}{cc}
X_{11}&X_{12}\\
X_{21}&X_{22}
\end{array}
\right)$. Then $MM^{\#}=M^{\#}M$, and so
$$\left(
\begin{array}{cc}
E&I\\
F&0
\end{array}
\right)\left(
\begin{array}{cc}
X_{11}&X_{12}\\
X_{21}&X_{22}
\end{array}
\right)=\left(
\begin{array}{cc}
X_{11}&X_{12}\\
X_{21}&X_{22}
\end{array}
\right)\left(
\begin{array}{cc}
E&I\\
F&0
\end{array}
\right).$$ Then we have
$$\begin{array}{c}
EX_{11}+X_{21}=X_{11}E+X_{12}F,\\
FX_{12}=X_{21}.
\end{array}$$
Since $MM^{\#}M=M$, we have
$$\begin{array}{c}
EX_{11}+X_{21}=I,\\
FX_{11}=0.
\end{array}$$
Therefore $$\begin{array}{lll}
F&=&(FE)X_{11}+FX_{21}\\
&=&(FE)F^dFX_{11}+FX_{21}\\
&=&(FEF^d)(FX_{11})+FX_{21}\\
&=&F^2X_{12}.
\end{array}$$ In view of~\cite[Lemma 1.2]{Z}, $F$ has group inverse.

Let $e=
\left(
\begin{array}{cc}
FF^{\#}&0\\
0&I
\end{array}
\right)$. Then $M=
\left(
\begin{array}{cc}
a&b\\
c&d
\end{array}
\right)_e,$ where $$\begin{array}{rll}
a&=&\left(
\begin{array}{cc}
FF^{\#}E&FF^{\#}\\
F^2F^{\#}&0
\end{array}
\right),b=\left(
\begin{array}{cc}
0&0\\
FF^{\pi}&0
\end{array}
\right)=0,\\
c&=&\left(
\begin{array}{cc}
F^{\pi}EFF^{\#}&F^{\pi}\\
0&0
\end{array}
\right),d=\left(
\begin{array}{cc}
EF^{\pi}&0\\
0&0
\end{array}
\right).
\end{array}$$
As in the proof of Theorem 2.5, we prove that $(EF^{\pi})^D=E^DF^{\pi}$. Then we compute that
$$a^{\#}=\left(
\begin{array}{cc}
0&F^{\#}\\
FF^{\#}&-FF^{\#}EF^{\#}
\end{array}
\right), d^D=\left(
\begin{array}{cc}
E^DF^{\pi}&0\\
0&0
\end{array}
\right).$$ Therefore we have $$a^{\pi}=\left(
\begin{array}{cc}
0&0\\
0&F^{\pi}
\end{array}
\right), d^{\pi}=\left(
\begin{array}{cc}
E^{\pi}F^{\pi}&0\\
0&0
\end{array}
\right).$$  In view of ~\cite[Theorem 2.1]{MD}, we have $d^{\pi}ca^{\pi}=0$, and so
$$\left(
\begin{array}{cc}
E^{\pi}F^{\pi}&0\\
0&0
\end{array}
\right)\left(
\begin{array}{cc}
F^{\pi}EFF^D&F^{\pi}\\
0&0
\end{array}
\right)\left(
\begin{array}{cc}
0&0\\
0&F^{\pi}
\end{array}
\right)=\left(
\begin{array}{cc}
0&E^{\pi}F^{\pi}\\
0&0
\end{array}
\right)=0.$$ Therefore $E^{\pi}F^{\pi}=0$, as required.

$(2)\Rightarrow (1)$ Since $EF^{\pi}$ had g-Drazin inverse and $FEF^{\pi}=0$, it follows by~\cite[Lemma 2.2]{Zhang} that $FF^{\#}E$ has g-Drazin inverse and
$(FF^{\#}E)^d=FF^{\#}E^d=FF^{\#}E^D$. Then $FE^DF^{\pi}=F(FF^{\#}E^D)F^{\pi}=F(FF^{\#}E)^dF^{\pi}=F[(FF^{\#}E)^d]^2F^{\#}(FEF^{\pi})=0$.

Set $$N=\left(
\begin{array}{cc}
E^DF^{\pi}&F^{\#}+(E^DF^{\pi})^2-E^DF^{\pi}EF^{\#}\\
FF^{\#}&-FF^{\#}EF^{\#}
\end{array}
\right).$$ Then we directly check that
$$\begin{array}{rll}
MN&=&\left(
\begin{array}{cc}
E&I\\
F&0
\end{array}
\right)\left(
\begin{array}{cc}
E^DF^{\pi}&F^{\#}+(E^DF^{\pi})^2-E^DF^{\pi}EF^{\#}\\
FF^{\#}&-FF^{\#}EF^{\#}
\end{array}
\right)\\
&=&\left(
\begin{array}{cc}
EE^DF^{\pi}+FF^{\#}&FF^{\#}EF^{\#}+EE^DF^{\pi}E^DF^{\pi}\\
0&FF^{\#}
\end{array}
\right)\\
&=&\left(
\begin{array}{cc}
I&E^DF^{\pi}\\
0&FF^{\#}
\end{array}
\right)\\
&=&\left(
\begin{array}{cc}
E^DF^{\pi}EF^{\pi}+FF^{\#}&E^DF^{\pi}\\
FF^{\#}EF^{\pi}&FF^{\#}
\end{array}
\right)\\
&=&\left(
\begin{array}{cc}
E^DF^{\pi}&F^{\#}+(E^DF^{\pi})^2-E^DF^{\pi}EF^{\#}\\
FF^{\#}&-FF^{\#}EF^{\#}
\end{array}
\right)\left(
\begin{array}{cc}
E&I\\
F&0
\end{array}
\right)\\
&=&NM,\\
\end{array}$$
$$\begin{array}{rll}
M(1-MN)&=&\left(
\begin{array}{cc}
E&I\\
F&0
\end{array}
\right)\left(
\begin{array}{cc}
0&-E^DF^{\pi}\\
0&F^{\pi}
\end{array}
\right)=0,\\
\end{array}$$
$$\begin{array}{rl}
&(1-MN)N\\
=&\left(
\begin{array}{cc}
0&-E^DF^{\pi}\\
0&F^{\pi}
\end{array}
\right)\left(
\begin{array}{cc}
E^DF^{\pi}&F^{\#}+(E^DF^{\pi})^2-E^DF^{\pi}EF^{\#}\\
FF^{\#}&-FF^{\#}EF^{\#}
\end{array}
\right)\\
=&0.
\end{array}$$
Therefore $M^{\#}=N$, as asserted.\end{proof}

\begin{cor} Let $M=\left(
\begin{array}{cc}
E&F\\
I&0
\end{array}
\right)$ and $E,F,EF^{\pi}$ have Drazin inverses. If $FEF^{\pi}=0$, then the following are equivalent:
\end{cor}
\begin{enumerate}
\item [(1)]{\it $M$ has group inverse.} \vspace{-.5mm}
\item [(2)]{\it $F$ has group inverse and $E^{\pi}F^{\pi}=0$.}\\
\end{enumerate}
In this case, $$M^{\#}=\left(
  \begin{array}{cc}
    \Gamma&\Delta\\
    \Lambda&\Xi\\
     \end{array}
\right),$$
where $$\begin{array}{lll}
\Gamma&=&F^{\pi}E^DF^{\pi},\\
\Delta&=&I-F^{\pi}E^DF^{\pi}E,\\
\Lambda&=&F^{\#}+(E^DF^{\pi})^2-E^DF^{\pi}EF^{\#},\\
\Xi&=&E^DF^{\pi}-F^{\#}E-(E^DF^{\pi})^2E+E^DF^{\pi}EF^{\#}E,
\end{array}$$
\begin{proof} Let $N=\left(
\begin{array}{cc}
E&I\\
F&0
\end{array}
\right)$. Then $$M=P^{-1}NP, P=\left(
\begin{array}{cc}
0&I\\
I&-E
\end{array}
\right).$$ Therefore $M$ has group inverse if and only if so does $N$, if and only if $F$ has group inverse and $E^{\pi}F^{\pi}=0$, by Theorem 3.1.
In this case, $$\begin{array}{lll}
M^{\#}&=&P^{-1}N^{\#}P\\
 &=&\left(
\begin{array}{cc}
E&I\\
I&0
\end{array}
\right)N^{\#}\left(
\begin{array}{cc}
0&I\\
I&-E
\end{array}
\right)\\
&=&\left(
\begin{array}{cc}
I&F^{\pi}E^DF^{\pi}\\
E^DF^{\pi}&F^{\#}+(E^DF^{\pi})^2-E^DF^{\pi}EF^{\#}
\end{array}
\right)\left(
\begin{array}{cc}
0&I\\
I&-E
\end{array}
\right)\\
&=&\left(
  \begin{array}{cc}
    \Gamma&\Delta\\
    \Lambda&\Xi\\
     \end{array}
\right),
\end{array}$$
where $$\begin{array}{lll}
\Gamma&=&F^{\pi}E^DF^{\pi},\\
\Delta&=&I-F^{\pi}E^DF^{\pi}E,\\
\Lambda&=&F^{\#}+(E^DF^{\pi})^2-E^DF^{\pi}EF^{\#},\\
\Xi&=&E^DF^{\pi}-F^{\#}E-(E^DF^{\pi})^2E+E^DF^{\pi}EF^{\#}E,
\end{array}$$ as asserted.\end{proof}

\begin{thm} Let $M=\left(
\begin{array}{cc}
E&F\\
I&0
\end{array}
\right)$ and $E,F,EF^{\pi}$ have Drazin inverse. If $F^{\pi}EF=0$, then the following are equivalent:
\end{thm}
\begin{enumerate}
\item [(1)]{\it $M$ has group inverse.} \vspace{-.5mm}
\item [(2)]{\it $F$ has group inverse and $F^{\pi}E^{\pi}=0$.}\\
\end{enumerate}
In this case, $$M^{\#}=\left(
\begin{array}{cc}
F^{\pi}E^D&FF^{\#}\\
F^{\#}+(F^{\pi}E^D)^2-F^{\#}EF^{\pi}E^D&-F^{\#}EFF^{\#}
\end{array}
\right).$$
\begin{proof} We consider the transpose $M^T=\left(
\begin{array}{cc}
E^T&I\\
F^T&0
\end{array}
\right)$ of $M$. Then $M$ has group inverse if and only if so does $M^T$.
Applying Theorem 3.1, $M$ has group inverse if and only if $F^T$ has group inverse and $(E^T)^{\pi}(F^T)^{\pi}=0$, i.e.,
$F$ has group inverse and $F^{\pi}E^{\pi}=0$. In this case, we have
$$\begin{array}{l}
M^{\#}=[(M^T)^{\#}]^T=\\
{\small \left(
\begin{array}{cc}
(E^T)^D(F^T)^{\pi}&(F^T)^{\#}+((E^T)^D(F^T)^{\pi})^2-(E^T)^D(F^T)^{\pi}E^T(F^T)^{\#}\\
F^T(F^T)^{\#}&-F^T(F^T)^{\#}E^T(F^T)^{\#}
\end{array}
\right)^T},
\end{array}$$
as desired.\end{proof}

\begin{cor} Let $M=\left(
\begin{array}{cc}
E&I\\
F&0
\end{array}
\right)$ and $E,F$ have Drazin inverses. If $F^{\pi}EF=0$, then the following are equivalent:
\end{cor}
\begin{enumerate}
\item [(1)]{\it $M$ has group inverse.} \vspace{-.5mm}
\item [(2)]{\it $F$ has group inverse and $F^{\pi}E^{\pi}=0$.}\\
\end{enumerate}
In this case, $$M^{\#}=\left(
  \begin{array}{cc}
    \Gamma&\Delta\\
    \Lambda&\Xi\\
     \end{array}
\right),$$
where $$\begin{array}{lll}
\Gamma&=&F^{\pi}E^DF^{\pi},\\
\Delta&=&F^{\#}+(F^{\pi}E^D)^2-F^{\#}EF^{\pi}E^D,\\
\Lambda&=&I-EF^{\pi}E^DF^{\pi},\\
\Xi&=&F^{\pi}E^D-EF^{\#}-E(F^{\pi}E^D)^2+EF^{\#}EF^{\pi}E^D,
\end{array}$$
\begin{proof} Let $N=\left(
\begin{array}{cc}
E&F\\
I&0
\end{array}
\right)$. Then $$M=P^{-1}NP, P=\left(
\begin{array}{cc}
E&I\\
I&0
\end{array}
\right).$$ In view of Theorem 3.3, $$N^{\#}=\left(
\begin{array}{cc}
F^{\pi}E^D&FF^{\#}\\
F^{\#}+(F^{\pi}E^D)^2-F^{\#}EF^{\pi}E^D&-F^{\#}EFF^{\#}
\end{array}
\right).$$ Hence, $M$ has group inverse if and only if so does $N$, if and only if $F$ has group inverse and $F^{\pi}E^{\pi}=0$, by Theorem 3.3.
Moreover, we have $$\begin{array}{lll}
M^{\#}&=&P^{-1}N^{\#}P\\
 &=&\left(
\begin{array}{cc}
0&I\\
I&-E
\end{array}
\right)N^{\#}\left(
\begin{array}{cc}
E&I\\
I&0
\end{array}
\right)\\
&=&\left(
  \begin{array}{cc}
    \Gamma&\Delta\\
    \Lambda&\Xi\\
     \end{array}
\right),
\end{array}$$
where $$\begin{array}{lll}
\Gamma&=&F^{\pi}E^DF^{\pi},\\
\Delta&=&F^{\#}+(F^{\pi}E^D)^2-F^{\#}EF^{\pi}E^D,\\
\Lambda&=&I-EF^{\pi}E^DF^{\pi},\\
\Xi&=&F^{\pi}E^D-EF^{\#}-E(F^{\pi}E^D)^2+EF^{\#}EF^{\pi}E^D,
\end{array}$$ as asserted. \end{proof}

We come now to extend ~\cite[Theorem 2.10]{Z} to wider cases as follows.

\begin{cor} Let $M=\left(
\begin{array}{cc}
E&F\\
I&0
\end{array}
\right)$ and $E,F,EF^{\pi}$ have Drazin inverses. If $EF=\lambda FE (\lambda \in {\Bbb C})$ or $EF^2=FEF$, then the following are equivalent:
\end{cor}
\begin{enumerate}
\item [(1)]{\it $M$ has group inverse.} \vspace{-.5mm}
\item [(2)]{\it $F$ have group inverse and $F^{\pi}E^{\pi}=0$.}\\
\end{enumerate}
In this case, $$M^{\#}=\left(
\begin{array}{cc}
F^{\pi}E^D&FF^{\#}\\
F^{\#}+(F^{\pi}E^D)^2-F^{\#}EF^{\pi}E^D&-F^{\#}EFF^{\#}
\end{array}
\right).$$
\begin{proof} If $EF=\lambda FE (\lambda \in {\Bbb C})$, then $F^{\pi}EF=\lambda F^{\pi}FE=0$. If $EF^2=FEF$, then
$F^{\pi}EF=F^{\pi}EF^2F^{\#}=F^{\pi}FEFF^{\#}=0$. This completes the proof by Theorem 3.3.\end{proof}

\section{block-operator matrices with identical subblocks}

In ~\cite{C1}, Cao et al. considered the group inverse for block matrices with identical subblocks over a right Ore domain. In this section we are concerned with
the group inverse for block-operator matrices with identical subblocks over a Banach space.

\begin{thm} Let $M=\left(
\begin{array}{cc}
E&F\\
F&0
\end{array}
\right)$ and $E,EF^{\pi}$ have Drazin inverse and $F$ has group inverse. If $FEF^{\pi}=0$, then the following are equivalent:
\end{thm}
\begin{enumerate}
\item [(1)]{\it $M$ has group inverse.} \vspace{-.5mm}
\item [(2)]{\it $EE^{\pi}F^{\pi}=0$.}
\end{enumerate}
In this case, $$M^{\#}=\left(
  \begin{array}{cc}
    \Gamma&\Delta\\
    \Lambda&\Xi\\
     \end{array}
\right),$$
where $$\begin{array}{rll}
\Gamma&=&[I-E^{\pi}F^{\pi}][E^DF^{\pi}+E^{\pi}F^{\pi}E(F^{\#})^2]+E^{\pi}F^{\pi}E(F^{\#})^2,\\
\Delta&=&[I-E^{\pi}F^{\pi}][F^{\#}-E^{\pi}F^{\pi}E(F^{\#})^2EF^{\#}-E^DF^{\pi}EF^{\#}]\\
&-&E^{\pi}F^{\pi}E(F^{\#})^2EF^{\#},\\
\Lambda&=&F[E^DF^{\pi}+E^{\pi}F^{\pi}E(F^{\#})^2]^2+F^{\#}-FE^{\pi}F^{\pi}[E(F^{\#})^2]^2\\
&-&FE^DF^{\pi}E(F^{\#})^2,\\
\Xi&=&[FE^DF^{\pi}+FE^{\pi}F^{\pi}E(F^{\#})^2][F^{\#}-E^{\pi}F^{\pi}E(F^{\#})^2EF^{\#}\\
&-&E^DF^{\pi}EF^{\#}]-[F^{\#}-FE^{\pi}F^{\pi}E(F^{\#})^2E(F^{\#})^2\\
&-&FE^DF^{\pi}E(F^{\#})^2]EF^{\#}.
\end{array}$$
\begin{proof}
$(1)\Rightarrow (2)$ Obviously, we have
$$\begin{array}{c}
M=\left(
\begin{array}{cc}
F^{\pi}E&F\\
F&0
\end{array}
\right)\left(
\begin{array}{cc}
I&I\\
F^{\#}E&0
\end{array}
\right),\\
M^2=\left(
\begin{array}{cc}
EF^{\pi}E+F^2&EF\\
0&F^2
\end{array}
\right)\left(
\begin{array}{cc}
I&I\\
F^{\#}E&0
\end{array}
\right).
\end{array}$$
Write $M^{\#}=\left(
\begin{array}{cc}
X_{11}&X_{12}\\
X_{21}&X_{22}
\end{array}
\right)$. Then $M^{\#}M^2=M$, and so
$$\left(
\begin{array}{cc}
X_{11}&X_{12}\\
X_{21}&X_{22}
\end{array}
\right)\left(
\begin{array}{cc}
EF^{\pi}E+F^2&EF\\
0&F^2
\end{array}
\right)=\left(
\begin{array}{cc}
F^{\pi}E&F\\
F&0
\end{array}
\right).$$ Therefore $$X_{11}EF^{\pi}E+X_{11}F^2=F^{\pi}E,$$ hence,
$$X_{11}EF^{\pi}EF^{\pi}+X_{11}F^2F^{\pi}=F^{\pi}EF^{\pi}.$$ It follows that
$$X_{11}(EF^{\pi})^2=F^{\pi}EF^{\pi}=EF^{\pi}.$$ In light of ~\cite[Lemma 1.2]{Z},
$EF^{\pi}$ has group inverse. By virtue of~\cite[Lemma 2.2]{Zhang}, $(EF^{\pi})^{\#}=E^DF^{\pi}.$ Therefore
$$\begin{array}{lll}
EF^{\pi}&=&(EF^{\pi})^{\#}EF^{\pi}EF^{\pi}\\
&=&E^DF^{\pi}EF^{\pi}EF^{\pi}\\
&=&EE^DF^{\pi};
\end{array}$$ hence, $EE^{\pi}F^{\pi}=0$.

$(2)\Rightarrow (1)$ Since $EE^{\pi}F^{\pi}=0$, we have $E^DF^{\pi}(EF^{\pi})^2=E^2E^DF^{\pi}=E(1-E^{\pi})F^{\pi}=
EF^{\pi}$, and so $EF^{\pi}$ has group inverse by~\cite[Lemma 1.2]{Z}.

Let $N=\left(
\begin{array}{cc}
E&I\\
F^2&0
\end{array}
\right)$. Choose $e=\left(
\begin{array}{cc}
FF^{\#}&0\\
0&I
\end{array}
\right).$ Then
$$a=\left(
\begin{array}{cc}
FF^{\#}E&FF^{\#}\\
F^2&0
\end{array}
\right),c=\left(
\begin{array}{cc}
F^{\pi}EFF^{\#}&F^{\pi}\\
0&0
\end{array}
\right),d=\left(
\begin{array}{cc}
EF^{\pi}&0\\
0&0
\end{array}
\right)$$ and $b=0$. Then
$$N=\left(
\begin{array}{cc}
a&0\\
c&d
\end{array}
\right)_e.$$
Moreover, we have $$a^{\#}=\left(
\begin{array}{cc}
0&(F^{\#})^2\\
FF^{\#}&-FF^{\#}E(F^{\#})^2
\end{array}
\right), d^{\#}=\left(
\begin{array}{cc}
E^DF^{\pi}&0\\
0&0
\end{array}
\right).$$ We compute that $$a^{\pi}=\left(
\begin{array}{cc}
0&0\\
0&F^{\pi}
\end{array}
\right), d^{\pi}=\left(
\begin{array}{cc}
E^{\pi}F^{\pi}&0\\
0&I
\end{array}
\right).$$ Obviously, $a^{\pi}cd^{\pi}=0$. In light of~\cite[Theorem 2.1]{MD}, $N$ has group inverse.
Moreover, we have $$N^{\#}=\left(
\begin{array}{cc}
a^{\#}&0\\
z&d^{\#}
\end{array}
\right),$$ where $$z=d^{\pi}c(a^{\#})^2+(d^{\#})^2ca^{\pi}-d^{\#}ca^{\#}.$$
Clearly, $$\begin{array}{rll}
d^{\pi}c(a^{\#})^2&=&\left(
\begin{array}{cc}
E^{\pi}F^{\pi}E(F^{\#})^2&-E^{\pi}F^{\pi}E(F^{\#})^2E(F^{\#})^2\\
0&0
\end{array}
\right),\\
(d^{\#})^2ca^{\pi}&=&\left(
\begin{array}{cc}
0&E^DF^{\pi}E^DF^{\pi}\\
0&0
\end{array}
\right), d^{\#}ca^{\#}=\left(
\begin{array}{cc}
0&E^DF^{\pi}E(F^{\#})^2\\
0&0
\end{array}
\right).
\end{array}$$
Hence we compute that $z=(z_{ij})$, where
$$\begin{array}{lll}
z_{11}&=&E^{\pi}F^{\pi}E(F^{\#})^2,\\
z_{12}&=&E^DF^{\pi}E^DF^{\pi}-E^{\pi}F^{\pi}E(F^{\#})^2E(F^{\#})^2-E^DF^{\pi}E(F^{\#})^2,\\
z_{21}&=&0, z_{22}=0.
\end{array}$$
Therefore $$N^{\#}=\left(
\begin{array}{cc}
\alpha&\beta\\
\gamma&\delta
\end{array}
\right),$$
where $$\begin{array}{lll}
\alpha&=&E^DF^{\pi}+E^{\pi}F^{\pi}E(F^{\#})^2,\\
\beta&=&(F^{\#})^2+E^DF^{\pi}E^DF^{\pi}-E^{\pi}F^{\pi}E(F^{\#})^2E(F^{\#})^2-E^DF^{\pi}E(F^{\#})^2,\\
\gamma&=&FF^{\#},\\
\delta&=&-FF^{\#}E(F^{\#})^2.
\end{array}$$
Hence, we have
$$\begin{array}{ll}
&NN^{\#}=N^{\#}N\\
=&\left(
\begin{array}{cc}
E&I\\
F^2&0
\end{array}
\right)\left(
\begin{array}{cc}
\alpha&\beta\\
\gamma&\delta
\end{array}
\right)\\
=&\left(
\begin{array}{cc}
\alpha&\beta\\
\gamma&\delta
\end{array}
\right)\left(
\begin{array}{cc}
E&I\\
F^2&0
\end{array}
\right)\\
=&\left(
\begin{array}{cc}
EE^DF^{\pi}+FF^{\#}&\alpha\\
F^2\alpha&FF^{\#}
\end{array}
\right).
\end{array}$$ Thus we have
$$N^{\pi}=\left(
\begin{array}{cc}
E^{\pi}F^{\pi}&-\alpha\\
-F^2\alpha&F^{\pi}
\end{array}
\right).$$
Obviously, one checks that
$$\begin{array}{c}
M=\left(
\begin{array}{cc}
E&I\\
F&0
\end{array}
\right)\left(
\begin{array}{cc}
I&0\\
0&F
\end{array}
\right),\\
N=\left(
\begin{array}{cc}
I&0\\
0&F
\end{array}
\right)\left(
\begin{array}{cc}
E&I\\
F&0
\end{array}
\right).
\end{array}$$
By virtue of Cline's formula, $M$ has Drazin inverse. We see that
$$\begin{array}{ll}
&\left(
\begin{array}{cc}
E&I\\
F&0
\end{array}
\right)N^{\pi}\left(
\begin{array}{cc}
I&0\\
0&F
\end{array}
\right)\\
=&\left(
\begin{array}{cc}
E&I\\
F&0
\end{array}
\right)\left(
\begin{array}{cc}
E^{\pi}F^{\pi}&-\alpha\\
-F^2\alpha&F^{\pi}
\end{array}
\right)\left(
\begin{array}{cc}
I&0\\
0&F
\end{array}
\right)\\
=&\left(
\begin{array}{cc}
-F^2\alpha&-E\alpha F\\
FE^{\pi}F^{\pi}&-F\alpha F
\end{array}
\right).
\end{array}$$ We directly compute that
$$\begin{array}{rll}
F^2\alpha&=&F^2E^DF^{\pi}+F^2E^{\pi}F^{\pi}E(F^{\#})^2\\
&=&F^2EF^{\pi}(E^DF^{\pi})^2+F^2(EF^{\pi})(E^DF^{\pi})E(F^{\#})^2=0,\\
E\alpha F&=&EE^{\pi}F^{\pi}EF^{\#}=0,\\
FE^{\pi}F^{\pi}&=&FE^DEF^{\pi}=F(E^DF^{\pi})(EF^{\pi})=F(EF^{\pi})(E^DF^{\pi})=0,\\
F\alpha F&=&FE^{\pi}F^{\pi}EF^{\#}=F(E^DF^{\pi})(EF^{\pi})EF^{\#}\\
&=&F(EF^{\pi})(E^DF^{\pi})EF^{\#}=0.
\end{array}$$ Hence $M=MM^DM$, i.e., $M$ has group inverse. Thus we have $M^{\#}=M^D$.

Moreover, we have $$\begin{array}{lll}
M^{D}&=&\left(
\begin{array}{cc}
E&I\\
F&0
\end{array}
\right)(N^{\#})^2\left(
\begin{array}{cc}
I&0\\
0&F
\end{array}
\right)\\
&=&\left(
\begin{array}{cc}
E&I\\
F&0
\end{array}
\right)\left(
\begin{array}{cc}
\alpha&\beta\\
\gamma&\delta
\end{array}
\right)^2\left(
\begin{array}{cc}
I&0\\
0&F
\end{array}
\right)\\
&=&\left(
  \begin{array}{cc}
    \Gamma&\Delta\\
    \Lambda&\Xi\\
     \end{array}
\right),
\end{array}$$
where $$\begin{array}{rll}
\Gamma&=&(E\alpha+\gamma)\alpha+(E\beta+\delta)\gamma,\\
\Delta&=&(E\alpha+\gamma)\beta F+(E\beta+\delta)\delta F,\\
\Lambda&=&F(\alpha^2+\beta\gamma),\\
\Xi&=&F(\alpha\beta +\beta\delta)F.
\end{array}$$ Therefore we complete the proof by the direct computation.\end{proof}

\begin{cor} Let $M=\left(
\begin{array}{cc}
E&F\\
F&0
\end{array}
\right)$ and $E,EF^{\pi}$ have Drazin inverse and $F$ has group inverse. If $F^{\pi}EF=0$, then the following are equivalent:\end{cor}
\begin{enumerate}
\item [(1)]{\it $M$ has group inverse.} \vspace{-.5mm}
\item [(2)]{\it $F^{\pi}E^{\pi}E=0$.}
\end{enumerate}
In this case, $$M^{\#}=\left(
  \begin{array}{cc}
    \Gamma&\Delta\\
    \Lambda&\Xi\\
     \end{array}
\right),$$
where $$\begin{array}{rll}
\Gamma&=&[F^{\pi}E^D+(F^{\#})^2EF^{\pi}E^{\pi}][I-F^{\pi}E^{\pi}]+(F^{\#})^2EF^{\pi}E^{\pi},\\
\Delta&=&[F^{\#}-F^{\#}E(F^{\#})^2EF^{\pi}E^{\pi}-F^{\#}EF^{\pi}E^D][I-F^{\pi}E^{\pi}]\\
&-&F^{\#}E(F^{\#})^2EF^{\pi}E^{\pi},\\
\Lambda&=&[F^{\pi}E^D+(F^{\#})^2EF^{\pi}E^{\pi}]^2F+F^{\#}-[(F^{\#})^2E]^2F^{\pi}E^{\pi}F\\
&-&(F^{\#})^2EF^{\pi}E^DF,\\
\Xi&=&[F^{\#}-F^{\#}E(F^{\#})^2EF^{\pi}E^{\pi}-F^{\#}EF^{\pi}E^D]\\
&&[F^{\pi}E^DF+(F^{\#})^2EF^{\pi}E^{\pi}F]-F^{\#}E[F^{\#}\\
&-&(F^{\#})^2E(F^{\#})^2EF^{\pi}E^{\pi}F-(F^{\#})^2EF^{\pi}E^DF].
\end{array}$$
\begin{proof} By virtue of Cline's formula, $F^{\pi}E$ has Drazin inverse. Then the proof is complete by applying Theorem 4.1 to the transpose $M^T=\left(
\begin{array}{cc}
E^T&F^T\\
F^T&0
\end{array}
\right).$\end{proof}

\begin{cor} Let $M=\left(
\begin{array}{cc}
E&F\\
F&0
\end{array}
\right)$ and $E,F$ have group inverse, $EF^{\pi}$ has Drazin inverse. If $F^{\pi}EF=0$, then $M$ has group inverse. In this case, $$M^{\#}=\left(
  \begin{array}{cc}
    \Gamma&\Delta\\
    \Lambda&\Xi\\
     \end{array}
\right),$$
where $$\begin{array}{rll}
\Gamma&=&[I-E^{\pi}F^{\pi}][E^{\#}F^{\pi}+E^{\pi}F^{\pi}E(F^{\#})^2]+E^{\pi}F^{\pi}E(F^{\#})^2,\\
\Delta&=&[I-E^{\pi}F^{\pi}][F^{\#}-E^{\pi}F^{\pi}E(F^{\#})^2EF^{\#}-E^{\#}F^{\pi}EF^{\#}]\\
&-&E^{\pi}F^{\pi}E(F^{\#})^2EF^{\#},\\
\Lambda&=&F[E^{\#}F^{\pi}+E^{\pi}F^{\pi}E(F^{\#})^2]^2+F^{\#}-FE^{\pi}F^{\pi}[E(F^{\#})^2]^2\\
&-&FE^{\#}F^{\pi}E(F^{\#})^2,\\
\Xi&=&[FE^{\#}F^{\pi}+FE^{\pi}F^{\pi}E(F^{\#})^2][F^{\#}-E^{\pi}F^{\pi}E(F^{\#})^2EF^{\#}\\
&-&E^{\#}F^{\pi}EF^{\#}]-[F^{\#}-FE^{\pi}F^{\pi}E(F^{\#})^2E(F^{\#})^2\\
&-&FE^{\#}F^{\pi}E(F^{\#})^2]EF^{\#}.
\end{array}$$
\end{cor}
\begin{proof} Since $E$ has group inverse, we see that $EE^{\pi}=0$, and so $EE^{\pi}F^{\pi}=0$. In light of Theorem 4.1, $M$ has group inverse. Therefore we obtain the representation of $M^{\#}$ by the formula in Theorem 4.1. \end{proof}

As an immediate consequence of Corollary 4.3, we have

\begin{cor} Let $M=\left(
\begin{array}{cc}
E&F\\
F&0
\end{array}
\right)$ and $E,F$ have group inverse, $EF^{\pi}$ has Drazin inverse. If $EF=\lambda FE ~(\lambda \in {\Bbb C})$ or $EF^2=FEF$, then $M$ has group inverse. In this case, $$M^{\#}=\left(
  \begin{array}{cc}
    \Gamma&\Delta\\
    \Lambda&\Xi\\
     \end{array}
\right),$$
where $$\begin{array}{rll}
\Gamma&=&[I-E^{\pi}F^{\pi}][E^{\#}F^{\pi}+E^{\pi}F^{\pi}E(F^{\#})^2]+E^{\pi}F^{\pi}E(F^{\#})^2,\\
\Delta&=&[I-E^{\pi}F^{\pi}][F^{\#}-E^{\pi}F^{\pi}E(F^{\#})^2EF^{\#}-E^{\#}F^{\pi}EF^{\#}]\\
&-&E^{\pi}F^{\pi}E(F^{\#})^2EF^{\#},\\
\Lambda&=&F[E^{\#}F^{\pi}+E^{\pi}F^{\pi}E(F^{\#})^2]^2+F^{\#}-FE^{\pi}F^{\pi}[E(F^{\#})^2]^2\\
&-&FE^{\#}F^{\pi}E(F^{\#})^2,\\
\Xi&=&[FE^{\#}F^{\pi}+FE^{\pi}F^{\pi}E(F^{\#})^2][F^{\#}-E^{\pi}F^{\pi}E(F^{\#})^2EF^{\#}\\
&-&E^{\#}F^{\pi}EF^{\#}]-[F^{\#}-FE^{\pi}F^{\pi}E(F^{\#})^2E(F^{\#})^2\\
&-&FE^{\#}F^{\pi}E(F^{\#})^2]EF^{\#}.
\end{array}$$
\end{cor}
\begin{proof} As in proof of Corollary 3.5, we obtain the result by Corollary 4.3.\end{proof}

We illustrate Theorem 4.1 by a numerical example.

\begin{exam} Let $M=\left(
\begin{array}{cc}
E&F\\
F&0
\end{array}
\right)\in {\Bbb C}^{4\times 4}$, where $E=\left(
\begin{array}{cc}
1& 2\\
0&-1
\end{array}
\right),F=\left(
\begin{array}{cc}
i& i\\
0&0
\end{array}
\right)\in {\Bbb C}^{2\times 2},$ $i^2=-1$. Then $$M^{\#}=\left(
\begin{array}{cccc}
0&1&-i&-i\\
0&-1&0&0\\
-i&-i&1&1\\
0&0&0&0
\end{array}
\right).$$\end{exam}
\begin{proof} Obviously, we have $$\begin{array}{rll}
E^{\#}&=&\left(
\begin{array}{cc}
1&2\\
0&-1
\end{array}
\right), E^{\pi}=0;\\
F^{\#}&=&\left(
\begin{array}{cc}
-i& -i\\
0&0
\end{array}
\right), F^{\pi}=\left(
\begin{array}{cc}
0&-1\\
0&1
\end{array}
\right).
\end{array}$$ Hence we check that $FEF^{\pi}=0, EE^{\pi}F^{\pi}=0.$ Construct $\Gamma, \Delta, \Lambda$ and $\Xi$ as in Theorem 4.1. Then we compute that
$$\begin{array}{rll}
\Gamma&=&\left(
\begin{array}{cc}
0&1\\
0&-1
\end{array}
\right), \Delta=\left(
\begin{array}{cc}
-i&-i\\
0&0
\end{array}
\right),\\
\Lambda&=&\left(
\begin{array}{cc}
-i&-i\\
0&0
\end{array}
\right), \Xi=\left(
\begin{array}{cc}
1&1\\
0&0
\end{array}
\right).
\end{array}$$
This completes the proof by Theorem 4.1.\end{proof}

\end{document}